\documentclass[pdflatex,sn-mathphys-num]{sn-jnl}

\usepackage{amsmath,amssymb,amsfonts}
\usepackage{amsthm}
\usepackage{mathrsfs}
\usepackage{graphicx}
\usepackage{xcolor}
\usepackage{tikz}
\usepackage{subcaption}
\usetikzlibrary{calc,backgrounds}

\theoremstyle{thmstyleone}
\newtheorem{thm}{Theorem}
\newtheorem{lem}{Lemma}
\newtheorem{cor}{Corollary}

\theoremstyle{thmstylethree}

\begin{document}

\title[Double Traversals in Boundary Subaisles]
{Double Traversals in Boundary Subaisles:
Implications~for~Two-Block~Layouts}

\author*[1]{\fnm{George} \sur{Dunn}}\email{george.dunn@uon.edu.au}
\author[1]{\fnm{Elizabeth} \sur{Stojanovski}}
\author[1]{\fnm{Bishnu} \sur{Lamichhane}}
\author[2]{\fnm{Hadi} \sur{Charkhgard}}
\author[3,4]{\fnm{Ali} \sur{Eshragh}}

\affil*[1]{\orgdiv{School of Information and Physical Sciences}, \orgname{University~of~Newcastle}, \orgaddress{\state{NSW}, \country{Australia}}}
\affil[2]{\orgdiv{Department of Industrial and Management Systems Engineering}, \orgname{University of South Florida}, \orgaddress{\state{FL}, \country{USA}}}
\affil[3]{\orgdiv{Carey Business School}, \orgname{Johns Hopkins University}, \orgaddress{\state{MD}, \country{USA}}}
\affil[4]{\orgdiv{International Computer Science Institute}, \orgname{University~of~California~at~Berkeley}, \orgaddress{\state{CA}, \country{USA}}}

\abstract{
The order picking problem seeks the shortest warehouse route that visits
all required item locations. Strict conditions are known
for single-block rectangular layouts under which optimal routes 
never require double traversals,
while broader results show that double traversals serving
cross-aisle connectivity can always be avoided.
We strengthen these findings by proving that no double traversals
are needed in the boundary subaisles,
the uppermost and lowermost subaisle segments,
of warehouses with at least two non-empty aisles.
This yields a unified strict condition for all single-block layouts
and for two-block layouts with more than one aisle.
For these widely used layouts,
exact methods such as dynamic programming and mathematical programming
can therefore exclude the double-traversal configuration
from every boundary subaisle,
reducing the number of admissible edge configurations
without loss of optimality.
}



\maketitle

\section{Introduction}
\label{introduction}

The order picking problem asks for the shortest route through a 
warehouse that visits all required item locations.
Dynamic programming algorithms for rectangular layouts
have been developed for two cross-aisles
\cite{ratliff1983order} and generalized to larger warehouses
\cite{roodbergen2001routing, pansart2018exact},
based on a restricted set of subaisle edge configurations.
\citet{revenant2025note} proved that in single-block warehouses, 
minimal routes can always be found without double traversals, 
establishing a strict condition under which double edges
are unnecessary.

Building on this,
our recent work \cite{dunn2025double} generalized this analysis to 
rectangular warehouses of arbitrary size.
By proposing a transformation framework for the underlying graph representation,
we showed that while double edges may still occur in minimal routes,
they are never required for cross-aisle connectivity, and that
any minimal route can be represented with only single edges
connecting components.
However, that result left open the question of when
non-connecting double edges are structurally unavoidable,
particularly in multi-block layouts.

In this paper, we close this remaining gap for the most
practically relevant layouts.
We prove that no double edges of any kind are required
in the boundary subaisles of a rectangular warehouse
with at least two non-empty aisles.
Since single-block and two-block layouts consist only of such
boundary subaisles, this yields a unified strict condition under
which optimal routes never require double traversals in
\emph{any} single-block layout and in all two-block layouts
with more than one aisle.
This establishes an entirely new strict result for two-block
warehouses and consolidates the earlier single-block condition
within a broader structural framework.

Beyond their structural interest, these results have direct
algorithmic implications.
Dynamic programming algorithms for rectangular warehouses
\cite{ratliff1983order, roodbergen2001routing, pansart2018exact}
and edge-based mathematical programming formulations
\cite{goeke2021modeling}
enumerate six vertical edge configurations at each subaisle,
causing the total number of candidate configurations to grow
combinatorially with warehouse size.
By eliminating the double-traversal configuration
from boundary subaisles,
our results reduce this combinatorial space
warehouse-wide for single-block and two-block layouts,
and at every boundary subaisle for larger multi-block layouts.

\section{Prerequisites}
\label{prerequisites}

This paper extends the framework introduced in \citet{dunn2025double}.
As we adopt the same warehouse terminology and graph representations,
only the key concepts needed for the present results are summarized below.
The reader is referred to \citet{dunn2025double}
for detailed definitions and figures.

A rectangular warehouse consists of a sequence of vertical aisles
separated by horizontal cross-aisles,
forming a grid containing vertices at each intersection.
A subaisle refers to the aisle segment between
two consecutive cross-aisles
and is where items are stored.
The picker routing problem seeks a minimum-length tour
of the graph that visits all locations
on a specified pick list.
The depot is the only pick location
that may coincide with an intersection vertex.

\citet{ratliff1983order} showed that for a minimal tour,
only six possible vertical edge configurations
are valid in each subaisle,
including single and double traversals.
These six configurations form the building blocks
for dynamic programming algorithms
\cite{ratliff1983order, roodbergen2001routing, pansart2018exact}
and edge-based mathematical programming formulations
\cite{goeke2021modeling}
in rectangular warehouses.
\citet{dunn2025double} defined a single edge
$(a_{i}, b_{i})^1$
as a single traversal spanning one or more subaisles
without intermediate horizontal edges.
A double edge $(a_{i}, b_{i})^2$
is defined analogously as a pair of parallel edges
spanning the same subaisles.
A double edge with horizontal edges incident
to both end vertices is called ``connecting''
and was shown to be unnecessary.

\begin{thm}[\citet{dunn2025double}]
\label{thm:connecting} 
There exists a minimum length tour subgraph
$T \subseteq G$ that does not contain a
connecting double edge.
\end{thm}

An aisle is said to be non-empty if it contains
at least one pick location.
When a layout with two or more blocks
has exactly one non-empty aisle,
a double edge is trivially required to connect the depot
to the farthest item.
Hence, we focus on layouts
with two or more non-empty aisles.

In \citet{dunn2025double} a state representation
for double edges was provided as
$s = (s_a, s_b)$,
defined by the number of horizontal
edges incident to the left of each end vertex.
Due to the horizontal and vertical symmetry of the graph,
there are only five distinct states
$\mathcal{S} = \{ (0,1), (0,2), (1,1), (1,2), (2,2) \}$.
This representation applies to non-connecting double edges
with horizontal connections on one end,
for which only states $(0,1)$ and $(0,2)$ are valid.
The state $(0, 0)$ is not required
because we consider warehouses with
two or more non-empty aisles and
one end of the non-connecting
double edge must have an incident horizontal edge
to ensure connectivity to the rest of the graph.
These are illustrated in Fig.~\ref{fig:transformations}
along with the corresponding transformation
$T \rightarrow T'$,
which maps one tour subgraph to another by
shifting one vertical edge
from the current aisle to the previous aisle
and one horizontal edge from the bottom to the top
\cite{dunn2025double}.
Intuitively, this operation ``slides'' edges across
adjacent aisles and cross-aisles while preserving feasibility
and tour length.

\begin{figure}[ht]
\begin{center}
\centering
    \begin{subfigure}[b]{0.47\textwidth}
        \centering
        \resizebox{\linewidth}{!}{
            \begin{tikzpicture}[shorten >=1pt,draw=black!50]
                \pgfmathsetmacro{\x}{2}
                \pgfmathsetmacro{\y}{4}
                \pgfmathsetmacro{\b}{0.5}
                \pgfmathsetmacro{\nodesize}{30}
                \node[circle, draw=black, minimum size=\nodesize pt, fill=white] (a_j-1) at (0, \y) { };
                \node[align=center] at (a_j-1) {$a_{i-1}$};
                \node[circle, draw=black, minimum size=\nodesize pt, fill=white] (a_j) at (\x, \y) { };
                \node[align=center] at (a_j) {$a_i$};
                \node[circle, draw=black, minimum size=\nodesize pt, fill=white] (b_j) at (\x, 0) { };
                \node[align=center] at (b_j) {$b_i$};
                \node[circle, draw=black, minimum size=\nodesize pt, fill=white] (b_j-1) at (0, 0) { };
                \node[align=center] at (b_j-1) {$b_{i-1}$};
                \draw[-, draw=black, thick, double, double distance between line centers=8pt] (a_j) -- (b_j);
                \draw[-, draw=black, thick] (b_j-1) -- (b_j);
                \begin{pgfonlayer}{background}
                    \fill[gray!20]
                    ($(a_j-1) + (\b, 0)$) -- ($(b_j-1) + (\b, 0)$) --
                    ($(b_j-1) + (-\b, 0)$) -- ($(a_j-1) + (-\b, 0)$) -- cycle;
                \end{pgfonlayer}
                \draw[->, draw=black, thick] (1.4*\x, 0.5*\y) -- (1.6*\x, 0.5*\y);
                \node[circle, draw=black, minimum size=\nodesize pt, fill=white] (a2_j-1) at (2*\x, \y) { };
                \node[align=center] at (a2_j-1) {$a_{i-1}$};
                \node[circle, draw=black, minimum size=\nodesize pt, fill=white] (a2_j) at (3*\x, \y) { };
                \node[align=center] at (a2_j) {$a_i$};
                \node[circle, draw=black, minimum size=\nodesize pt, fill=white] (b2_j) at (3*\x, 0) { };
                \node[align=center] at (b2_j) {$b_i$};
                \node[circle, draw=black, minimum size=\nodesize pt, fill=white] (b2_j-1) at (2*\x, 0) { };
                \node[align=center] at (b2_j-1) {$b_{i-1}$};
                \draw[-, draw=black, thick] (a2_j) -- (b2_j);
                \draw[-, draw=red, thick] (a2_j-1) -- (b2_j-1);
                \draw[-, draw=black, thick] (a2_j-1) -- (a2_j);
                \begin{pgfonlayer}{background}
                    \fill[gray!20]
                    ($(a2_j-1) + (\b, 0)$) -- ($(b2_j-1) + (\b, 0)$) --
                    ($(b2_j-1) + (-\b, 0)$) -- ($(a2_j-1) + (-\b, 0)$) -- cycle;
                \end{pgfonlayer}
            \end{tikzpicture}
        }
        \caption{$s=(0, 1)$}
        \label{fig:transformation_01}
    \end{subfigure}
    \hfill
    \begin{subfigure}[b]{0.47\textwidth}
        \centering
        \resizebox{\linewidth}{!}{
            \begin{tikzpicture}[shorten >=1pt,draw=black!50]
                \pgfmathsetmacro{\x}{2}
                \pgfmathsetmacro{\y}{4}
                \pgfmathsetmacro{\b}{0.5}
                \pgfmathsetmacro{\nodesize}{30}
                \node[circle, draw=black, minimum size=\nodesize pt, fill=white] (a_j-1) at (0, \y) { };
                \node[align=center] at (a_j-1) {$a_{i-1}$};
                \node[circle, draw=black, minimum size=\nodesize pt, fill=white] (a_j) at (\x, \y) { };
                \node[align=center] at (a_j) {$a_i$};
                \node[circle, draw=black, minimum size=\nodesize pt, fill=white] (b_j) at (\x, 0) { };
                \node[align=center] at (b_j) {$b_i$};
                \node[circle, draw=black, minimum size=\nodesize pt, fill=white] (b_j-1) at (0, 0) { };
                \node[align=center] at (b_j-1) {$b_{i-1}$};
                \draw[-, draw=black, thick, double, double distance between line centers=8pt] (a_j) -- (b_j);
                \draw[-, draw=black, thick, double, double distance between line centers=8pt] (b_j-1) -- (b_j);
                \begin{pgfonlayer}{background}
                    \fill[gray!20]
                    ($(a_j-1) + (\b, 0)$) -- ($(b_j-1) + (\b, 0)$) --
                    ($(b_j-1) + (-\b, 0)$) -- ($(a_j-1) + (-\b, 0)$) -- cycle;
                \end{pgfonlayer}
                \draw[->, draw=black, thick] (1.4*\x, 0.5*\y) -- (1.6*\x, 0.5*\y);
                \node[circle, draw=black, minimum size=\nodesize pt, fill=white] (a2_j-1) at (2*\x, \y) { };
                \node[align=center] at (a2_j-1) {$a_{i-1}$};
                \node[circle, draw=black, minimum size=\nodesize pt, fill=white] (a2_j) at (3*\x, \y) { };
                \node[align=center] at (a2_j) {$a_i$};
                \node[circle, draw=black, minimum size=\nodesize pt, fill=white] (b2_j) at (3*\x, 0) { };
                \node[align=center] at (b2_j) {$b_i$};
                \node[circle, draw=black, minimum size=\nodesize pt, fill=white] (b2_j-1) at (2*\x, 0) { };
                \node[align=center] at (b2_j-1) {$b_{i-1}$};
                \draw[-, draw=black, thick] (a2_j) -- (b2_j);
                \draw[-, draw=red, thick] (a2_j-1) -- (b2_j-1);
                \draw[-, draw=black, thick] (a2_j-1) -- (a2_j);
                \draw[-, draw=black, thick] (b2_j-1) -- (b2_j);
                \begin{pgfonlayer}{background}
                    \fill[gray!20]
                    ($(a2_j-1) + (\b, 0)$) -- ($(b2_j-1) + (\b, 0)$) --
                    ($(b2_j-1) + (-\b, 0)$) -- ($(a2_j-1) + (-\b, 0)$) -- cycle;
                \end{pgfonlayer}
            \end{tikzpicture}
        }
        \caption{$s=(0, 2)$}
        \label{fig:transformation_02}
    \end{subfigure}
\caption{Transformation $T \rightarrow T'$
for each non-connecting double edge state.
Gray shading indicates arbitrary feasible configurations,
while the red line denotes the addition of a single edge
to these configurations.}
\label{fig:transformations}
\end{center}
\end{figure}

Within the proof of Theorem~\ref{thm:connecting},
\citet{dunn2025double} showed that double edges in states
$(0,2)$, $(1,1)$, $(1,2)$ and $(2,2)$
can always be reduced using this transformation.
We record this here as a supporting lemma.

\begin{lem}
\label{lem:connecting}
Double edges belonging to
states $(0,2)$, $(1,1)$, $(1,2)$ and $(2,2)$
are unnecessary.
\end{lem}

A second condition from \citet{dunn2025double} identifies
when the transformation fails to reduce a double edge.

\begin{lem}
\label{lem:single}
Under transformation $T \rightarrow T'$,
the double edge $(a_{i-1}, b_{i-1})^2 \in T$ is
always reduced unless the preceding subaisles
contain a single edge $(a_{i-1}, b_{i-1})^1 \in T$.
\end{lem}

Together, these results imply that,
in any warehouse with two or more non-empty aisles,
the only double edges that might still be required
are non-connecting edges in state $(0,1)$
whose preceding aisle segment contains a single edge.
The present paper shows that even these
remaining double edges are unnecessary
whenever they lie in a boundary subaisle.

\section{Main Result}
\label{proof_v3}

We extend the framework of \citet{dunn2025double}
to eliminate all double edges in the boundary
subaisles of any rectangular warehouse
with two or more non-empty aisles.
The \emph{boundary subaisles} are the uppermost and lowermost
subaisles (see Fig.~\ref{fig:structural_boundary}).

\begin{figure}[ht]
\centering
\begin{tikzpicture}[
    scale=1,
    every node/.style={circle, draw=black, minimum size=6mm},
    execute at end picture={\path[use as bounding box] (-3, 0.7) rectangle (9, -8);}]
\clip (-3, 0.7) rectangle (9, -8);

\def\nAisles{5}
\def\nRows{6}

\tikzset{
  dbl/.style={-, draw=black, double,
  double distance=5pt
  }
}

\fill[gray!20] (0.5,-0.5) rectangle (\nAisles+0.5,-2.25);
\fill[gray!20] (0.5,-4.75) rectangle (\nAisles+0.5,-6.5);

\foreach \i in {1,...,\nAisles} {
    \foreach \j in {1,...,\nRows} {
        \node[fill=white] (v\i\j) at (\i,-\j) {};
    }
}

\foreach \i in {1,...,\nAisles} {
    \foreach \j in {1,...,\numexpr\nRows-1\relax} {
        \pgfmathtruncatemacro{\nextj}{\j+1}
        \draw[dbl] (v\i\j) -- (v\i\nextj);
    }
}

\foreach \j in {1,...,\nRows} {
    \foreach \i in {1,...,\numexpr\nAisles-1\relax} {
        \pgfmathtruncatemacro{\nexti}{\i+1}
        \draw[dbl] (v\i\j) -- (v\nexti\j);
    }
}

\node[draw=none, align=center] (label) at (-2,-3.25) {boundary\\subaisles};
\draw[->, thick] (label.east) -- (0, -2);
\draw[->, thick] (label.east) -- (0, -5);

\node[draw=none] (front) at (\nAisles/2+0.5, -7.5) {front of warehouse};
\node[draw=none] (back) at (\nAisles/2+0.5, 0.5) {back of warehouse};

\draw[dashed] (0.5, 0) -- (\nAisles+0.5, 0);
\draw[dashed] (0.5, -7) -- (\nAisles+0.5, -7);

\end{tikzpicture}
\caption{Boundary subaisles (shaded) in a rectangular warehouse graph.
Our main result shows that no double edge in these subaisles is required
for a minimal tour when there are two or more non-empty aisles.}
\label{fig:structural_boundary}
\end{figure}

\begin{thm}[Elimination of double edges in boundary subaisles]
\label{thm:main} 
In any rectangular warehouse
with two or more non-empty aisles,
a double edge in the boundary subaisles
is not required for a minimal tour.
\end{thm}

\begin{proof}
We begin by making explicit what we mean by a feasible \emph{tour
subgraph} $T\subseteq G$.
Following \citet{ratliff1983order}, feasibility means:
(i) $T$ contains every required pick location, including the depot,
(ii) $T$ is connected,
and (iii) every vertex in $T$ has even degree.
The length of $T$ is the sum of the lengths of all its edges.
The transformation $T\rightarrow T'$ from \citet{dunn2025double}
replaces a local edge pattern by another that retains feasibility
and has the same total edge length, so the resulting tour subgraph
$T'$ is feasible and has the same length as $T$.

From Theorem~\ref{thm:connecting}, connecting double edges are
unnecessary, so it suffices to consider non-connecting double edges.
For these, only states $(0,1)$ and $(0,2)$ can occur, and
Lemma~\ref{lem:connecting} shows that double edges in state $(0,2)$ are not
required.
Lemma~\ref{lem:single} further ensures that applying the transformation
$T\rightarrow T'$ always reduces double edges unless the preceding aisle
segment has a single edge.
Consequently, the only boundary-subaisle double edge that could
remain in a minimum-length tour subgraph is a non-connecting double
edge in state $(0,1)$ whose preceding aisle segment contains a
single edge.
Fig.~\ref{fig:double} depicts this configuration and the
transformation applied to it.

A non-connecting double edge in a boundary subaisle has one
endpoint on the warehouse border (top or bottom), which we refer to
as the \emph{outer endpoint}.
We call such a double edge \emph{inner-connected} if the outer endpoint
has no incident horizontal edges in $T$, and \emph{outer-connected}
otherwise.
By symmetry, the arguments for the upper and lower boundary subaisles
are identical.
Fig.~\ref{fig:double} illustrates the inner-connected case in the upper
subaisle and the outer-connected case in the lower subaisle to match the
order of the state representation.

Assume, for contradiction, that every minimum-length tour subgraph
contains at least one boundary-subaisle double edge, and among
these choose $T$ with the fewest such double edges.
We show that in every possible configuration we can obtain either
(i) a strictly shorter feasible tour subgraph or
(ii) a feasible tour subgraph of the same length with fewer
boundary-subaisle double edges,
contradicting the choice of $T$.
The argument proceeds by three exhaustive cases.

\begin{figure}[ht]
\begin{center}
\centering
    \begin{subfigure}[t]{0.49\textwidth}
        \centering
        \resizebox{\linewidth}{!}{
            \begin{tikzpicture}[shorten >=1pt,draw=black!50]

                \pgfmathsetmacro{\x}{1.5}
                \pgfmathsetmacro{\y}{3}
                \pgfmathsetmacro{\z}{0.5}

                \pgfmathsetmacro{\a}{0.25}

                \pgfmathsetmacro{\b}{0.5}

                \pgfmathsetmacro{\nodesize}{25}

                
                \node[circle, draw=black, minimum size=\nodesize pt, fill=white]
                (a_j) at (\x, \y) { };
                \node[align=center] at (a_j) {$a_i$};

                \node[circle, draw=black, minimum size=\nodesize pt, fill=white]
                (a_j-1) at (0, \y) { };
                \node[align=center] at (a_j-1) {$a_{i-1}$};

                \node[circle, draw=black, minimum size=\nodesize pt, fill=white]
                (a_j+1) at (2*\x, \y) { };
                \node[align=center] at (a_j+1) {$a_{i+1}$};

                \node[circle, draw=black, minimum size=\nodesize pt, fill=white]
                (b_j) at (\x, 0) { };
                \node[align=center] at (b_j) {$b_i$};

                \node[circle, draw=black, minimum size=\nodesize pt, fill=white]
                (b_j-1) at (0, 0) { };
                \node[align=center] at (b_j-1) {$b_{i-1}$};

                \node[circle, draw=black, minimum size=\nodesize pt, fill=white]
                (b_j+1) at (2*\x, 0) { };
                \node[align=center] at (b_j+1) {$b_{i+1}$};

                \node[circle, draw=black, minimum size=\nodesize pt, fill=white]
                (p1) at (\x, 0.5*\y) { };

                \draw[-, draw=black, thick, double, double distance between line centers=8pt]
                (p1) -- (b_j);

                \draw[-, draw=red, thick, double, double distance between line centers=8pt]
                (p1) -- (a_j);

                \draw[-, draw=black, thick]
                (b_j-1) -- (b_j);

                \draw[-, draw=black, thick]
                (b_j-1) -- (a_j-1);

                \begin{pgfonlayer}{background}
                    \fill[gray!20]
                    ($(b_j-1) + (-\b, 0)$) --
                    ($(b_j-1) + (-\b, -2*\b)$) --
                    ($(b_j+1) + (\b, -2*\b)$) --
                    ($(a_j+1) + (\b, 0)$) --
                    ($(a_j+1) + (-\b, 0)$) --
                    ($(b_j+1) + (-\b, \b)$) --
                    ($(b_j) + (0, \b)$) --
                    ($(b_j) + (0, 0)$) --
                    ($(b_j) + (-\b, 0)$) --
                    ($(b_j) + (-\b, -\b)$) --
                    ($(b_j-1) + (\b, -\b)$) --
                    ($(b_j-1) + (\b, 0)$) --
                    cycle;

                \end{pgfonlayer}

                \draw[-, dotted, draw=black, thick]
                ($(a_j-1) + (-\b, 2*\b)$) -- ($(a_j+1) + (\b, 2*\b)$);


                \draw[->, draw=black, thick] (2.5*\x + 0.2*\z, 0.5*\y)
                -- (2.5*\x + 0.8*\z, 0.5*\y);

                
                \node[circle, draw=black, minimum size=\nodesize pt, fill=white]
                (a2_j) at (4*\x + \z, \y) { };
                \node[align=center] at (a2_j) {$a_i$};

                \node[circle, draw=black, minimum size=\nodesize pt, fill=white]
                (a2_j-1) at (3*\x + \z, \y) { };
                \node[align=center] at (a2_j-1) {$a_{i-1}$};

                \node[circle, draw=black, minimum size=\nodesize pt, fill=white]
                (a2_j+1) at (5*\x + \z, \y) { };
                \node[align=center] at (a2_j+1) {$a_{i+1}$};

                \node[circle, draw=black, minimum size=\nodesize pt, fill=white]
                (b2_j) at (4*\x + \z, 0) { };
                \node[align=center] at (b2_j) {$b_i$};

                \node[circle, draw=black, minimum size=\nodesize pt, fill=white]
                (b2_j-1) at (3*\x + \z, 0) { };
                \node[align=center] at (b2_j-1) {$b_{i-1}$};

                \node[circle, draw=black, minimum size=\nodesize pt, fill=white]
                (b2_j+1) at (5*\x + \z, 0) { };
                \node[align=center] at (b2_j+1) {$b_{i+1}$};

                \node[circle, draw=black, minimum size=\nodesize pt, fill=white]
                (p2) at (4*\x + \z, 0.5*\y) { };

                \draw[-, draw=black, thick]
                (a2_j) -- (p2) -- (b2_j);

                \draw[-, draw=black, thick]
                (a2_j-1) -- (a2_j);

                \draw[-, draw=black, thick, double, double distance between line centers=8pt]
                (b2_j-1) -- (a2_j-1);

                \begin{pgfonlayer}{background}
                    \fill[gray!20]
                    ($(b2_j-1) + (-\b, 0)$) --
                    ($(b2_j-1) + (-\b, -2*\b)$) --
                    ($(b2_j+1) + (\b, -2*\b)$) --
                    ($(a2_j+1) + (\b, 0)$) --
                    ($(a2_j+1) + (-\b, 0)$) --
                    ($(b2_j+1) + (-\b, \b)$) --
                    ($(b2_j) + (0, \b)$) --
                    ($(b2_j) + (0, 0)$) --
                    ($(b2_j) + (-\b, 0)$) --
                    ($(b2_j) + (-\b, -\b)$) --
                    ($(b2_j-1) + (\b, -\b)$) --
                    ($(b2_j-1) + (\b, 0)$) --
                    cycle;

                \end{pgfonlayer}

                \draw[-, dotted, draw=black, thick]
                ($(a2_j-1) + (-\b, 2*\b)$) -- ($(a2_j+1) + (\b, 2*\b)$);


            \end{tikzpicture}
        }
        \caption{Inner-connected}
        \label{fig:double_inner}
    \end{subfigure}
    \begin{subfigure}[t]{0.49\textwidth}
        \centering
        \resizebox{\linewidth}{!}{
            \begin{tikzpicture}[shorten >=1pt,draw=black!50]

                \pgfmathsetmacro{\x}{1.5}
                \pgfmathsetmacro{\y}{3}
                \pgfmathsetmacro{\z}{0.5}

                \pgfmathsetmacro{\a}{0.25}

                \pgfmathsetmacro{\b}{0.5}

                \pgfmathsetmacro{\nodesize}{25}

                
                \node[circle, draw=black, minimum size=\nodesize pt, fill=white]
                (a_j) at (\x, \y) { };
                \node[align=center] at (a_j) {$a_i$};

                \node[circle, draw=black, minimum size=\nodesize pt, fill=white]
                (a_j-1) at (0, \y) { };
                \node[align=center] at (a_j-1) {$a_{i-1}$};

                \node[circle, draw=black, minimum size=\nodesize pt, fill=white]
                (a_j+1) at (2*\x, \y) { };
                \node[align=center] at (a_j+1) {$a_{i+1}$};

                \node[circle, draw=black, minimum size=\nodesize pt, fill=white]
                (b_j) at (\x, 0) { };
                \node[align=center] at (b_j) {$b_i$};

                \node[circle, draw=black, minimum size=\nodesize pt, fill=white]
                (b_j-1) at (0, 0) { };
                \node[align=center] at (b_j-1) {$b_{i-1}$};

                \node[circle, draw=black, minimum size=\nodesize pt, fill=white]
                (b_j+1) at (2*\x, 0) { };
                \node[align=center] at (b_j+1) {$b_{i+1}$};

                \draw[-, draw=black, thick, double, double distance between line centers=8pt]
                (a_j) -- (b_j);

                \draw[-, draw=black, thick]
                (b_j-1) -- (b_j) -- (b_j+1);

                \draw[-, draw=black, thick]
                (b_j-1) -- (a_j-1);

                \begin{pgfonlayer}{background}
                    \fill[gray!20]
                    ($(a_j-1) + (-\b, 0)$) --
                    ($(a_j-1) + (-\b, 2*\b)$) --
                    ($(a_j+1) + (\b, 2*\b)$) --
                    ($(b_j+1) + (\b, 0)$) --
                    ($(b_j+1) + (-\b, 0)$) --
                    ($(a_j+1) + (-\b, \b)$) --
                    ($(a_j) + (\b, \b)$) --
                    ($(a_j) + (\b, 0)$) --
                    ($(a_j) + (-\b, 0)$) --
                    ($(a_j) + (-\b, \b)$) --
                    ($(a_j-1) + (\b, \b)$) --
                    ($(a_j-1) + (\b, 0)$) --
                    cycle;

                \end{pgfonlayer}

                \draw[-, dotted, draw=black, thick]
                ($(b_j-1) + (-\b, -2*\b)$) -- ($(b_j+1) + (\b, -2*\b)$);


                \draw[->, draw=black, thick] (2.5*\x + 0.2*\z, 0.5*\y)
                -- (2.5*\x + 0.8*\z, 0.5*\y);

                
                \node[circle, draw=black, minimum size=\nodesize pt, fill=white]
                (a2_j) at (4*\x + \z, \y) { };
                \node[align=center] at (a2_j) {$a_i$};

                \node[circle, draw=black, minimum size=\nodesize pt, fill=white]
                (a2_j-1) at (3*\x + \z, \y) { };
                \node[align=center] at (a2_j-1) {$a_{i-1}$};

                \node[circle, draw=black, minimum size=\nodesize pt, fill=white]
                (a2_j+1) at (5*\x + \z, \y) { };
                \node[align=center] at (a2_j+1) {$a_{i+1}$};

                \node[circle, draw=black, minimum size=\nodesize pt, fill=white]
                (b2_j) at (4*\x + \z, 0) { };
                \node[align=center] at (b2_j) {$b_i$};

                \node[circle, draw=black, minimum size=\nodesize pt, fill=white]
                (b2_j-1) at (3*\x + \z, 0) { };
                \node[align=center] at (b2_j-1) {$b_{i-1}$};

                \node[circle, draw=black, minimum size=\nodesize pt, fill=white]
                (b2_j+1) at (5*\x + \z, 0) { };
                \node[align=center] at (b2_j+1) {$b_{i+1}$};

                \draw[-, draw=black, thick]
                (a2_j) -- (b2_j) -- (b2_j+1);

                \draw[-, draw=black, thick]
                (a2_j-1) -- (a2_j);

                \draw[-, draw=black, thick, double, double distance between line centers=8pt]
                (b2_j-1) -- (a2_j-1);

                \begin{pgfonlayer}{background}
                    \fill[gray!20]
                    ($(a2_j-1) + (-\b, 0)$) --
                    ($(a2_j-1) + (-\b, 2*\b)$) --
                    ($(a2_j+1) + (\b, 2*\b)$) --
                    ($(b2_j+1) + (\b, 0)$) --
                    ($(b2_j+1) + (-\b, 0)$) --
                    ($(a2_j+1) + (-\b, \b)$) --
                    ($(a2_j) + (\b, \b)$) --
                    ($(a2_j) + (\b, 0)$) --
                    ($(a2_j) + (-\b, 0)$) --
                    ($(a2_j) + (-\b, \b)$) --
                    ($(a2_j-1) + (\b, \b)$) --
                    ($(a2_j-1) + (\b, 0)$) --
                    cycle;

                \end{pgfonlayer}

                \draw[-, dotted, draw=black, thick]
                ($(b2_j-1) + (-\b, -2*\b)$) -- ($(b2_j+1) + (\b, -2*\b)$);


            \end{tikzpicture}
        }
        \caption{Outer-connected}
        \label{fig:double_outer}
    \end{subfigure}
\caption[Transformation for boundary subaisle double edges.]
{
Transformation $T \rightarrow T'$ for a non-connecting double edge
in a boundary subaisle with state $s=(0,1)$
(the only non-connecting state not excluded by Lemma~\ref{lem:connecting})
and a single edge $(a_{i-1},b_{i-1})^1$ in the preceding aisle segment
(the only situation in which Lemma~\ref{lem:single} does not already
reduce the double edge).
Gray regions indicate portions of the tour subgraph that may take any
feasible configuration (subject to even degrees and connectivity), and
the dotted line indicates the warehouse border.
(a) \emph{Inner-connected}: the outer endpoint has no incident horizontal
edges.
The unlabeled vertex is the nearest pick location to $a_i$ on the
boundary subaisle. If $a_i$ is not the depot, the red double segment can
be deleted, yielding a strictly shorter feasible tour subgraph.
(b) \emph{Outer-connected}: the outer endpoint is incident to a single
horizontal edge directed to the right.
This
incidence is forced in this state to maintain even degree at $b_i$.
}
\label{fig:double}
\end{center}
\end{figure}

\paragraph{Case 1: Inner-connected with \texorpdfstring{$a_i \neq$}{ai !=} depot.}
This is the configuration in Fig.~\ref{fig:double_inner}.
Because the boundary double edge is non-connecting and inner-connected,
the outer endpoint $a_i$ has no horizontal edges and
its only incident edges are the two parallel edges
between $a_i$ and the nearest pick location (the unlabeled
vertex).
Consequently, deleting \emph{both} edges in that segment
(highlighted in red)
decreases the degree of each of its endpoints by
$2$ and leaves all other vertex degrees unchanged, so all degrees remain
even.
The pick location remains connected to the rest of
the tour through the remaining double edge leading inward.
Since $a_i$ is not the depot (and is not required to be visited in this
case), removing it from the tour does not violate the requirement
to visit all required locations.
Thus we obtain a feasible tour subgraph of strictly smaller length,
contradicting that $T$ is minimum-length.

\paragraph{Case 2: Outer-connected.}
Fig.~\ref{fig:double_outer} shows this case and the resulting
transformation $T\rightarrow T'$, which produces an inner-connected
double edge $(a_{i-1},b_{i-1})^2$ in the adjacent aisle segment.
By Case~1, this is reducible unless its outer endpoint is the depot.
To see that a reducible direction always exists, note that state
$(0,1)$ gives $b_i$ a single horizontal edge to the left, and even
degree at $b_i$ forces a single horizontal edge to the right also.
The transformation can therefore be applied in either direction about aisle $i$.
Toward aisle $i-1$, the preceding segment contains a single edge
(as assumed), so $T'$ has an inner-connected double edge as above.
Toward aisle $i+1$, Lemma~\ref{lem:single} either directly reduces
the double edge (if no single edge is present in that segment)
or again yields an inner-connected double edge.
Since these two directions produce distinct outer endpoints and
there is only one depot, at least one direction gives an outer
endpoint that is not the depot, and Case~1 applies.

\paragraph{Case 3: Inner-connected with \texorpdfstring{$a_i =$}{ai =} depot.}
This is the remaining inner-connected possibility excluded from Case~1
only because we cannot delete the boundary segment incident to the depot.
Apply the transformation of Fig.~\ref{fig:double} to obtain $T'$.
In $T'$, the double edge becomes $(a_{i-1},b_{i-1})^2$.
If it is connecting, it is unnecessary by Theorem~\ref{thm:connecting}.
If it is non-connecting it is outer-connected and falls under Case~2 which can be eliminated.
In either subcase we obtain a feasible tour subgraph of length no greater
than $T$ with fewer boundary-subaisle double edges, contradicting the
choice of $T$.

\medskip

\noindent
All possible double edges in boundary subaisles are therefore
shown to be unnecessary.
\end{proof}

Single-block and two-block layouts consist
only of boundary subaisles (Fig.~\ref{fig:structural_twoblock}).
Hence Theorem~\ref{thm:main} immediately implies
that \emph{no} double traversals are required anywhere
in these layouts when there are two or more non-empty aisles.
For the remaining single-block case of exactly one non-empty aisle,
any double edge is necessarily inner-connected with the depot at
one outer endpoint.
The opposite outer endpoint is not the depot, so the boundary
segment on that side can be removed as in Case~1
(Fig.~\ref{fig:double_inner}), eliminating the double edge.
This recovers the strict condition of \citet{revenant2025note}
for all single-block warehouses regardless of the number of
non-empty aisles.
For two-block warehouses, which are widely used in practice
\cite{roodbergen2001routing}, the result is entirely new as
no strict structural result on double traversals has
previously been established for them.

\begin{figure}[ht]
\centering
\begin{tikzpicture}[
    scale=1,
    every node/.style={circle, draw=black, minimum size=6mm},
    execute at end picture={\path[use as bounding box] (-7, 0.7) rectangle (5, -5);}]
\clip (-7, 0.7) rectangle (5, -5);

\def\nAisles{4}
\def\nRows{3}

\tikzset{
  dbl/.style={-, draw=black, double,
  double distance=5pt
  }
}

\fill[gray!20] (-2.5,-0.5) rectangle (-\nAisles-2.5,-3.5);

\foreach \i in {1,...,\nAisles} {
    \foreach \j in {1,...,2} {
        \node[fill=white] (s\i\j) at (-2-\i,1-2*\j) {};
    }
}

\foreach \i in {1,...,\nAisles} {
    \foreach \j in {1,...,1} {
        \pgfmathtruncatemacro{\nextj}{\j+1}
        \draw[dbl] (s\i\j) -- (s\i\nextj);
    }
}

\foreach \j in {1,...,2} {
    \foreach \i in {1,...,\numexpr\nAisles-1\relax} {
        \pgfmathtruncatemacro{\nexti}{\i+1}
        \draw[dbl] (s\i\j) -- (s\nexti\j);
    }
}

\fill[gray!20] (0.5,-0.5) rectangle (\nAisles+0.5,-1.8);
\fill[gray!20] (0.5,-2.2) rectangle (\nAisles+0.5,-3.50);

\foreach \i in {1,...,\nAisles} {
    \foreach \j in {1,...,\nRows} {
        \node[fill=white] (v\i\j) at (\i,-\j) {};
    }
}

\foreach \i in {1,...,\nAisles} {
    \foreach \j in {1,...,\numexpr\nRows-1\relax} {
        \pgfmathtruncatemacro{\nextj}{\j+1}
        \draw[dbl] (v\i\j) -- (v\i\nextj);
    }
}

\foreach \j in {1,...,\nRows} {
    \foreach \i in {1,...,\numexpr\nAisles-1\relax} {
        \pgfmathtruncatemacro{\nexti}{\i+1}
        \draw[dbl] (v\i\j) -- (v\nexti\j);
    }
}

\draw[-, thick, dashed] (-\nAisles-2.5,-3.75) -- (\nAisles+0.5,-3.75);
\draw[-, thick, dashed] (-\nAisles-2.5,-0.25) -- (\nAisles+0.5,-0.25);

\node[draw=none, align=center] (label) at (-1,-2) {boundary\\subaisles};
\draw[->, thick] (label.east) -- (0.5, -1.2);
\draw[->, thick] (label.east) -- (0.5, -2.8);
\draw[->, thick] (label.west) -- (-2.5, -2);

\node[draw=none] (front) at (-1, -4.25) {front of warehouse};
\node[draw=none] (back) at (-1, 0.25) {back of warehouse};

\end{tikzpicture}
\caption{
In single-block (left) and two-block (right) rectangular warehouses,
all subaisles are boundary subaisles (shaded). Hence Theorem~\ref{thm:main}
implies that no double edges are required in a minimal tour for either
layout when there are two or more non-empty aisles (Corollary~\ref{cor:unify}).}
\label{fig:structural_twoblock}
\end{figure}

\begin{cor}
\label{cor:unify}
In a rectangular warehouse,
double edges are never required in
a minimal tour of either
(i) a single-block layout with any number of non-empty aisles,
or (ii) a two-block layout with two or more non-empty aisles.
\end{cor}

Corollary~\ref{cor:unify} guarantees that the
double-traversal configuration can be excluded
from every subaisle in single-block warehouses and in
two-block warehouses with two or more non-empty aisles,
reducing the number of admissible vertical configurations
from six to five at each subaisle without loss of optimality.
The case of a single non-empty aisle in a two-block layout is
trivial as the optimal tour is always a single double edge from the depot
to the farthest pick location.
Because these configurations are enumerated combinatorially
in dynamic programming
\cite{ratliff1983order, roodbergen2001routing, pansart2018exact}
and encoded as decision variables in
mathematical programming formulations
\cite{goeke2021modeling},
this warehouse-wide reduction directly yields
smaller state spaces and tighter models.

\section{Conclusion}
We have proved that double traversals are not required
in the boundary subaisles of rectangular warehouses
with at least two non-empty aisles,
establishing the first strict structural result
for two-block rectangular warehouses
and consolidating the known single-block condition
of \citet{revenant2025note}
within a broader framework.
For these layouts, the number of admissible vertical
configurations at each subaisle is reduced from six to five
without loss of optimality,
directly benefiting dynamic programming algorithms
and mathematical programming formulations
through smaller state spaces and fewer decision variables.
The characterization of necessary double edges,
which are now known to be limited to non-connecting doubles
in interior subaisles of warehouses with three or more blocks,
remains an open question for further research.

\clearpage

\section*{Declarations}

\paragraph{Funding}
George Dunn was supported by an Australian Government
Research Training Program (RTP) Scholarship. No other
funding was received for this work.

\paragraph{Competing interests}
The authors have no competing interests to declare
that are relevant to the content of this article.

\paragraph{Data availability}
This work is purely theoretical and does not use
empirical datasets. No data were generated or analysed.

\bibliography{bibliography}

\end{document}